\documentclass[12pt]{amsart}

\usepackage[utf8]{inputenc}%
\usepackage[T1]{fontenc}%
\usepackage{lmodern}
\usepackage[top=1.15in, bottom=1.15in, left=1.2in, right=1.2in]{geometry}
\usepackage{amsmath,amssymb}
\usepackage{amsthm}
\usepackage{mathtools}
\usepackage{mathrsfs} 
\usepackage{xfrac} 
\usepackage{dsfont} 
\usepackage{array}
\usepackage{verbatim}
\usepackage{graphicx}
\usepackage{enumitem} 
\usepackage{lipsum}%
\usepackage{relsize}
\usepackage{exscale}
\usepackage{tikz-cd}
\usepackage[hidelinks]{hyperref}
\theoremstyle{plain} 
\begingroup 
\newtheorem{thm}{Theorem}[section] 
\newtheorem{cor}[thm]{Corollary} 
\newtheorem{lem}[thm]{Lemma} 
\newtheorem{prop}[thm]{Proposition} 
\endgroup

\theoremstyle{definition}
\theoremstyle{remark}
\newtheorem{remark}[thm]{Remark} 
\newtheorem{example}[thm]{Example} 

\newtheoremstyle{ser}
{8pt}
{8pt}
{\it}
{}
{\sf\bfseries}
{:}
{6mm}
{}

\newtheoremstyle{serr}
{8pt}
{8pt}
{\normalfont}
{}
{\sf}
{.}
{6mm}
{}

\theoremstyle{ser}

\theoremstyle{serr}

\theoremstyle{ser}

\numberwithin{equation}{section}

     		\newcommand{\kB}{\mathscr{B}}

     		\newcommand{\kH}{\mathscr{H}}

     		\newcommand{\kV}{\mathscr{V}}
\newcommand{\fX}{\mathfrak X}

\def\XXint#1#2#3{{\setbox0=\hbox{$#1{#2#3}{\int}$ }
		\vcenter{\hbox{$#2#3$ }}\kern-.6\wd0}}

\newlength{\mylen}
\setbox1=\hbox{$\bullet$}\setbox2=\hbox{\tiny$\bullet$}
\newcommand{\R}{\mathbb{R}}
\newcommand{\Z}{\mathbb{Z}}
\newcommand{\C}{\mathbb{C}}

\newcommand{\N}{\mathbb{N}}

\newcommand{\D}{\mathbb{D}}

\newcommand{\ep}{\varepsilon}

\newcommand{\ud}{\,\mathrm{d}} 

\newcommand{\ran}{\textnormal{\textrm{ran}}}

\newcommand{\rP}[1]{\mathbf{R}\!\left( #1 \right)}

\usepackage[utf8]{inputenc}
\usepackage{graphicx} 

\newcommand{\Title}{On the convergence of the normalized power sequence of Riesz operators}%
\newcommand{\ShortTitle}{}

\makeatletter
\newcommand{\raisemath}[1]{\mathpalette{\raisem@th{#1}}}
\newcommand{\raisem@th}[3]{\raisebox{#1}{$#2#3$}}
\makeatother

\newcommand{\AuthorOne}{Soumyashant Nayak}%
\newcommand{\AuthorOneAddr}{%
Statistics and Mathematics Unit, Indian Statistical Institute, 8th Mile, Mysore Road, RVCE Post, Bengaluru and Karnataka - 560 059, India
}%
\newcommand{\AuthorOneEmail}{%
soumyashant@isibang.ac.in
}%

\newcommand{\AuthorTwo}{Renu Shekhawat}%
\newcommand{\AuthorTwoAddr}{%
Department of Mathematics, Indian Institute of Science, C. V. Raman Road, Bengaluru and Karnataka - 560 012, India
}%
\newcommand{\AuthorTwoEmail}{%
renushekhawa@iisc.ac.in
}%

\newcommand{\Keywords}{Normalized power sequence, Riesz operator, Yamamoto-Davis theorem}%

\newcommand{\SubjectClassification}{Primary 47B06; Secondary 47A10, 47A75}

\title[\MakeUppercase\ShortTitle]{\MakeUppercase \Title}
\author{\AuthorOne}%
\address[\AuthorOne]{\AuthorOneAddr}%
\email{\href{mailto:\AuthorOneEmail}{\AuthorOneEmail}}%

\author{\AuthorTwo}%
\address[\AuthorTwo]{\AuthorTwoAddr}%
\email{\href{mailto:\AuthorTwoEmail}{\AuthorTwoEmail}}%

\date{}%
\subjclass[2020]{\SubjectClassification}%

\begin{document}

\maketitle

\begin{abstract}
Let $\kH$ be a complex Hilbert space and $\kB(\kH)$ be the algebra of all bounded linear operators on $\kH$. For $T \in \kB(\kH)$, let $|T| := (T^*T)^{\frac{1}{2}}$. We refer to the sequence $\big\{|T^n|^{\frac{1}{n}}\big\}_{n \in \N}$ as the NPS (normalized power sequence) of $T$. In this article, we show that the NPS of a Riesz operator $R \in \kB(\kH)$ converges in norm to a positive operator $H$, and provide an explicit description of the spectral resolution of $H$ in terms of the Riesz idempotents associated with the non-zero eigenvalues of $R$. Since every compact operator is a Riesz operator, this gives us a stronger, spatial generalization of the Yamamoto-Davis theorem, which asserts that $\lim_{n \to \infty} s_j(K^n)^{\frac{1}{n}}$ is equal to the $j^{\textrm{th}}$-largest eigenvalue-modulus of the compact operator $K$, where $s_j(\cdot)$ denotes the $j^{\textrm{th}}$-largest singular value. In recent work, the present authors have established the norm convergence of the NPS for spectral operators. Using a rank-one perturbation of a unitary operator, we demonstrate that this fails in general for essentially spectral operators (of which Riesz operators form a subclass).
\end{abstract}

\section{Introduction}

For an operator $T \in \kB(\kH)$, let $|T| := ({T}^* T)^{\frac{1}{2}}$. We refer to the sequence $\big\{|T^n|^{\frac{1}{n}}\big\}_{n \in \N}$ as the normalized power sequence of $T$, henceforth referred to as the NPS of $T$. The classical spectral radius formula, $r(T) = \lim_{n \to \infty} \|T^n\|^{\frac{1}{n}} = \lim_{n \to \infty} \| |T^n|^{\frac{1}{n}} \|$, shows that the operator norm of the elements of the NPS converges to the spectral radius $r(T)$. However, characterizing the asymptotic behavior of the sequence itself requires a more delicate geometric and spectral analysis.

The investigation of such asymptotic behavior traces back to Yamamoto's theorem for matrix singular values (see \cite{yamamoto}). Although Yamamoto also established partial extensions for compact operators on infinite-dimensional separable Hilbert spaces (see \cite{yamamoto_compact}), Davis achieved the definitive generalization in \cite{davis}. Specifically, for a compact operator $K \in \kB(\kH)$, Davis proved that $\lim_{n \to \infty} s_j(K^n)^{\frac{1}{n}} = |\lambda_j|$ for all $j \in \N$, where $s_j(\cdot)$ denotes the $j^{\textrm{th}}$ largest singular value, and $|\lambda_j|$ denotes the $j^{\textrm{th}}$ largest eigenvalue-modulus of $K$ (counted according to algebraic multiplicity) with the convention that if the total multiplicity of the non-zero eigenvalues is a finite number $\nu(K) \in \mathbb{N}$, then $\lambda_j = 0$ for $j > \nu(K)$. To achieve this, he utilized Yamamoto's theorem (\cite{yamamoto}) alongside the machinery of exterior powers and compound operators to find the individual singular values. This was coupled with a spectral splitting technique to separate the part of the spectrum away from zero from the asymptotically vanishing remainder.

While these classical results focus on the scalar sequences of singular values, recent work has shifted toward establishing the convergence of the global operator sequence itself. In \cite{haagerup_schultz},  Haagerup and Schultz studied the SOT-convergence of the NPS of operators in $II_1$ factors. In \cite{stronger_form_yamamoto}, the first-named author established that the NPS of a finite-dimensional matrix converges in norm to a positive semi-definite matrix completely described by the diagonalizable part of the matrix in its Jordan-Chevalley decomposition. Subsequently, a stronger form of Davis' theorem was proved by Bhat and Bala in \cite{bhat-bala}, where they showed that the NPS of a compact operator acting on a separable Hilbert space converges in norm.
In \cite{nayak_shekhawat}, the present authors established the norm convergence of the NPS of a spectral operator (in the sense of Dunford, see \cite[\S 1]{Dunford_survey}) in $\kB(\kH)$, and the norm-limit was explicitly characterized using the idempotent-valued spectral resolution of the spectral operator.

Given that spectral and compact operators possess highly disciplined spectral structures, it is natural to explore whether these convergence properties extend to operators that spectrally mimic them.
Owing to the Riesz-Fredholm theory (see \cite{conway}), the non-zero spectrum of a compact operator consists entirely of isolated eigenvalues of finite algebraic multiplicity. Seeking the broadest possible class of bounded operators that preserve this spectral behaviour away from the origin, Ruston \cite[p.~318]{ruston1954} developed the theory of \textit{Riesz operators}.

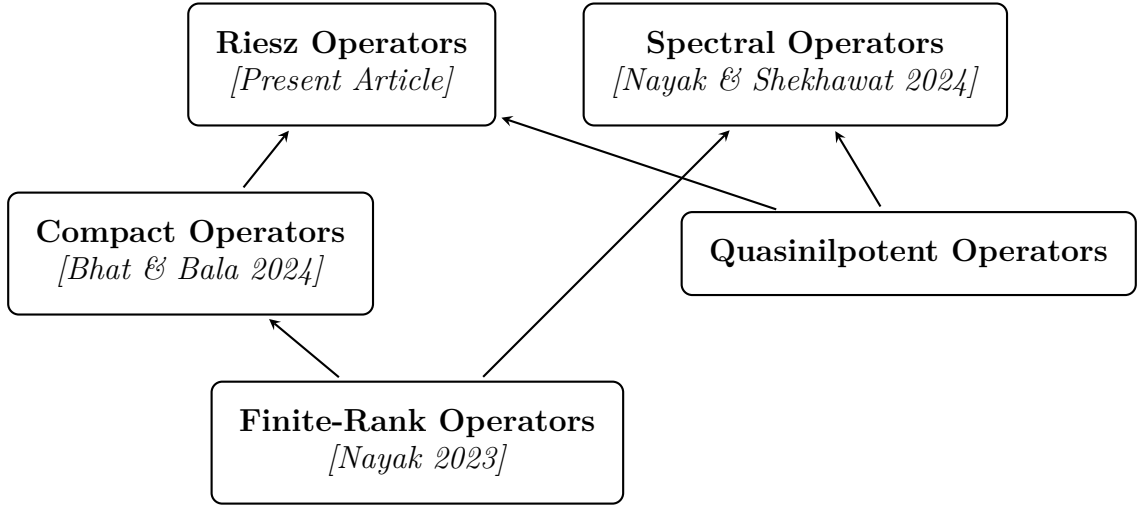
\begin{figure}[htpb]
\centering
\begin{tikzpicture}[
    >=stealth,
    box/.style={draw, thick, rounded corners, align=center, inner sep=2ex},
    inclusion/.style={->, thick, shorten >= 2pt, shorten <= 2pt}
]

\node[box] (riesz) at (-2.5, 5) {\textbf{Riesz Operators} \\ \textit{[Present Article]}};
\node[box] (spectral) at (3.5, 5) {\textbf{Spectral Operators} \\ \textit{[Nayak \& Shekhawat 2024]}};

\node[box] (compact) at (-4.5, 2.5) {\textbf{Compact Operators} \\ \textit{[Bhat \& Bala 2024]}};
\node[box] (qn) at (5, 2.5) {\textbf{Quasinilpotent Operators}};

\node[box] (fr) at (-1.5, 0) {\textbf{Finite-Rank Operators} \\ \textit{[Nayak 2023]}};

\draw[inclusion] (fr) -- (compact);
\draw[inclusion] (fr) -- (spectral);
\draw[inclusion] (compact) -- (riesz);
\draw[inclusion] (qn) -- (riesz);
\draw[inclusion] (qn) -- (spectral);

\end{tikzpicture}
\caption{Inclusion hierarchy of operator classes for which the norm-convergence of the normalized power sequence has been established. Arrows indicate subset inclusion.}
\label{fig:operator_hierarchy}
\end{figure}

In this article, we give a substantially simpler proof of norm-convergence as compared to that in \cite{bhat-bala}, in the more general setting of Riesz operators. Crucially, we obtain a concrete description of the spectral resolution of the limiting positive operator in terms of the Riesz idempotents associated with the non-zero eigenvalues, an aspect which is missing in \cite{bhat-bala}. The Yamamoto-Davis theorem follows as an immediate corollary (see Corollary \ref{cor:davis}) of our theorem.

To contextualize these developments, Figure~\ref{fig:operator_hierarchy} illustrates the inclusion hierarchy of the relevant operator classes for which the NPS converges in norm. As depicted, the class of finite-rank operators lies at the intersection of the distinct realms of compact and spectral operators. Moving up the hierarchy, the class of Riesz operators subsumes both compact and quasinilpotent operators. This containment is a direct consequence of their algebraic characterization from the work of Caradus \cite{caradus1966}: an operator $R \in \kB(\kH)$ is a Riesz operator if and only if its image under the quotient map in the Calkin algebra is quasinilpotent, that is, its essential spectrum is $\{0\}$.

We say that an operator in $\kB(\kH)$ is \textit{essentially spectral} if its image under the quotient map in the Calkin algebra is a spectral operator. Note that this class naturally extends both Riesz and spectral operators. In Example~\ref{ex:NPS_fail}, we provide an example of an essentially spectral operator---a rank-one perturbation of a unitary operator---whose NPS fails to converge in norm. Thus an exact characterization of operators in $\kB(\kH)$ possessing a norm-convergent NPS remains an open problem.

\vskip 0.1in

\noindent \textbf{Notation.}
Throughout this article, $\kH$ denotes a complex Hilbert space and $\kB(\kH)$ denotes the set of all bounded linear operators on $\kH$. Let $T$ be an operator in $\kB(\kH)$. We denote the spectrum and the essential spectrum of $T$ by $\sigma(T)$ and $\sigma_e(T)$, respectively, and the spectral radius of $T$ is denoted by $r(T)$. The range of $T$ is denoted by $\ran(T)$ and the orthogonal projection onto the closure of the range of $T$ is denoted by $\rP T$. For $r \ge 0$, we use $\D_{r}$ to denote the closed disc in $\C$ having radius $r$ centered at the origin, and for $r<0$, we stipulate that $\D_r:=\varnothing$.

\section{Main Results}
\label{sec:main}

We begin with a useful elementary lemma.

\begin{lem}
\label{lem:pos_inv_general}
\textsl{
For any bounded linear operator $T \in \kB(\kH)$, the operator $|T|^2 + |I-T|^2$ is positive and invertible.
}
\end{lem}
\begin{proof}
Let $P: = |T|^2 + |I-T|^2$. Clearly $P$ is positive. For any $x \in \kH$, we have $x = Tx + (I-T)x$. Applying the triangle inequality and the elementary inequality $(a+b)^2 \le 2(a^2+b^2)$, we obtain:
$$\|x\|^2 \le \big( \|Tx\| + \|(I-T)x\| \big)^2 \le 2 \big( \|Tx\|^2 + \|(I-T)x\|^2 \big) = 2 \langle Px, x \rangle.$$
Thus $P \ge \frac{1}{2} I$, which ensures invertibility. 
\end{proof}

Riesz operators are structurally characterized by a seminal result of West \cite[Theorem 7.5]{west}, which states that every Riesz operator $R \in \kB(\kH)$ admits a decomposition $R = K + Q$, where $K$ is a compact normal operator and $Q$ is a quasinilpotent operator. Moreover, this decomposition respects the spectral structure, as $\sigma(K) = \sigma(R)$ and the non-zero eigenvalues of $K$ and $R$ share the same algebraic multiplicities. This split echoes the Schur decomposition for compact operators on Hilbert space (see \cite{schur-decomp}), reinforcing the fact that Riesz operators structurally generalize compact operators. Unlike the Dunford decomposition (see \cite[Theorem XV.4.5]{dunford_schwartz_III}) of spectral operators ($T = D + Q$), the components $K$ and $Q$ in the West decomposition (also the components in the Schur decomposition) are generally not unique and \textit{do not necessarily commute}, nor are the associated Riesz idempotents guaranteed to be uniformly bounded.

This absence of commutativity and uniform boundedness of the idempotents means that the asymptotic techniques utilized in \cite{nayak_shekhawat} in the context of spectral operators require modification. While inspired by \cite[Lemma 4.3]{nayak_shekhawat}, our $\ep$-excision technique in the key lemma below relies on a different argument. 

By factoring out the Riesz idempotent corresponding to the $\ep$-disc, we asymptotically suppress the influence of the near-origin spectrum, thereby trapping the NPS of $R$ within a small neighborhood of the NPS of the finite-rank cut-off operator $R(I-e_{\ep})$, where $e_\ep$ denotes the Riesz idempotent of $R$ corresponding to the $\ep$-disc, $\D_\ep$.

\begin{lem}\label{lem:appx_ineq}
\textsl{
Let $R \in \kB(\kH)$ be a Riesz operator, and $0 < \ep < 1$. Let $e_{\ep}$ be the Riesz idempotent of $R$ corresponding to $\D_{\ep}$. Then there is a positive integer $n(\ep)$ such that for all $n \ge n(\ep)$, we have
$$(1-\ep) |(R(I-e_{\ep}))^n|^{\frac{1}{n}} \le |R^n|^{\frac{1}{n}} \le (1+\ep) |(R(I-e_{\ep}))^n|^{\frac{1}{n}} + 4 \ep I.$$
}
\end{lem}
\begin{proof}
Let $f_\ep := I - e_\ep$ and define $P_{\ep} := {f_\ep}^*{f_\ep} + e_{\ep}^*e_{\ep}$. By Lemma \ref{lem:pos_inv_general}, $P_{\ep}$ is an invertible positive operator, yielding the bounds $\|P_{\ep}^{-1}\|^{-1} I \le P_{\ep} \le \|P_{\ep}\| I$. Multiplying on the left by $(R^n)^*$ and on the right by $R^n$ establishes:
\begin{equation}
\label{eqn:P_bounds}
\begin{aligned}
    & \|P_{\ep}^{-1}\|^{-1} |R^n|^2 \le {R^n}^* P_{\ep} R^n \le \|P_{\ep}\| |R^n|^2 \\
    \Longrightarrow \quad & \|P_{\ep}\|^{-1} ({R^n}^* P_{\ep} R^n) \le |R^n|^2 \le \|P_{\ep}^{-1}\| ({R^n}^* P_{\ep} R^n).
\end{aligned}
\end{equation}

Since $R$ commutes with both $e_{\ep}$ and $f_{\ep}$, we have $${R^n}^* P_{\ep} R^n = |(R f_{\ep})^n|^2 + |(R e_{\ep})^n|^2.$$ Furthermore, $\sigma(Re_\ep) \subseteq \D_\ep$, so the spectral radius formula guarantees an $n_1(\ep) \in \N$ such that $\|(Re_\ep)^n\| \le (2\ep)^n$ for all $n \ge n_1(\ep)$, which implies $|(R e_{\ep})^n|^2 \le (2\ep)^{2n} I$. Substituting this yields:
\begin{equation}\label{eqn:combined_bounds}
    |(R f_{\ep})^n|^2 \le {R^n}^* P_{\ep} R^n \le |(R f_{\ep})^n|^2 + (2\ep)^{2n} I \quad \text{for all } n \ge n_1(\ep).
\end{equation}

Putting together (\ref{eqn:P_bounds}) and (\ref{eqn:combined_bounds}), we apply operator monotonicity of the function $x \mapsto x^{\frac{1}{2n}}$ on $[0, \infty)$, alongside the continuous functional calculus inequality $(H^n + \alpha^n I)^{\frac{1}{n}} \le H + \alpha I$ for a positive operator $H$, to obtain:
$$ \|P_{\ep}\|^{-\frac{1}{2n}} |(R f_{\ep})^n|^{\frac{1}{n}} \le |R^n|^{\frac{1}{n}} \le \|P_{\ep}^{-1}\|^{\frac{1}{2n}} \left( |(R f_{\ep})^n|^{\frac{1}{n}} + 2\ep I \right) \quad \text{for all } n \ge n_1(\ep). $$

Because $\lim_{n \to \infty} \|P_{\ep}\|^{-\frac{1}{2n}} = 1$ and $\lim_{n \to \infty} \|P_{\ep}^{-1}\|^{\frac{1}{2n}} = 1$, there exists an $n(\ep) \ge n_1(\ep)$ such that for all $n \ge n(\ep)$, we have that $\|P_{\ep}\|^{-\frac{1}{2n}}, \|P_{\ep}^{-1}\|^{\frac{1}{2n}}$ are bounded below by $1-\ep$, bounded above by $1+\ep$, respectively. Using the fact that $(1+\ep)(2\ep) = 2\ep + 2\ep^2 < 4\ep$ (since $\ep < 1$), for all $n \ge n(\ep)$ we arrive at the desired inequality:
\[
(1-\ep) |(R f_{\ep})^n|^{\frac{1}{n}} \le |R^n|^{\frac{1}{n}} \le (1+\ep) |(R f_{\ep})^n|^{\frac{1}{n}} + 4\ep I. \qedhere
\]
\end{proof}

Note that the spectrum of a finite-rank operator on an infinite-dimensional Hilbert space is finite and consists entirely of eigenvalues, necessarily including $0$.

\begin{prop}\label{prop:finite-rank}
\textsl{
Let $A \in \kB(\kH)$ be a finite-rank operator whose non-zero eigenvalues have moduli strictly greater than $\ep > 0$. For $\lambda \ge 0$, let $e_{\lambda}$ denote the Riesz idempotent of $A$ corresponding to $\D_{\lambda}$, and define $F_{\lambda} := \rP{e_{\lambda}}$. For $\lambda < 0$, define $F_{\lambda} = 0$. Then $\{F_{\lambda}\}_{\lambda \in \R}$ is a bounded resolution of the identity, and we have 
$$\lim_{n \to \infty} |A^n|^{\frac{1}{n}} = \int_{(\ep, r(A)]} \lambda \ud F_\lambda =: H \quad \text{in norm.}$$
The eigenvalues of $H$ are precisely the moduli of the eigenvalues of $A$. Furthermore, the algebraic multiplicity of any non-zero eigenvalue $\mu > 0$ of $H$ is exactly the sum of the algebraic multiplicities of all eigenvalues of $A$ having modulus $\mu$.
}
\end{prop}

\begin{proof}
Because $A$ is a finite-rank operator, the subspace $\kV := \ran(A) + \ran(A^*)$ is finite-dimensional. Clearly, $\kV$ reduces $A$, and $A$ acts as the zero operator on $\kV^{\perp}$. Restricting our attention to the finite-dimensional Hilbert space $\kV$, the operator $A|_{\kV}$ can be viewed as a square matrix. 

The existence of the norm limit of the NPS of  $A|_{\kV}$ and its explicit formulation in terms of the ranges of the Riesz idempotents of $A|_{\kV}$ (derived from its Jordan-Chevalley decomposition) follow directly from Theorem~3.8 in \cite{stronger_form_yamamoto}. Because the action of $A$ and its corresponding Riesz idempotents for non-zero eigenvalues are entirely supported on $\kV$, extending this finite-dimensional limit by zero to the orthogonal complement $\kV^{\perp}$ seamlessly yields the assertion on the entire space $\kH$.
\end{proof}

\begin{remark}
Because every finite-rank operator admits a Jordan--Chevalley decomposition, it is necessarily a spectral operator. Consequently, Proposition~\ref{prop:finite-rank} can also be deduced from \cite[Theorem~4.4]{nayak_shekhawat}. This alternative approach relies on a fundamentally different proof strategy than the one utilized in \cite[Theorem~3.8]{stronger_form_yamamoto}.
\end{remark}

We observe that for any $\ep > 0$, the part of the spectrum of a Riesz operator outside the disc $\D_{\ep}$ consists of at most finitely many isolated eigenvalues, each having finite algebraic multiplicity.

\begin{thm}
\label{thm:main}
\textsl{
Let $R \in \kB(\kH)$ be a Riesz operator. For $\lambda > 0$, let $e_{\lambda}$ denote the Riesz idempotent of $R$ corresponding to $\D_{\lambda}$, and define $F_{\lambda} := \rP{e_{\lambda}}$. Let $F_0 := \bigwedge_{\lambda' > 0} F_{\lambda'}$, and $F_{\lambda} = 0$ for $\lambda < 0$. Then $\{ F_{\lambda} \}_{\lambda \in \R}$ is a bounded resolution of the identity, and 
\[
    \lim_{n \to \infty} |R^n|^{\frac{1}{n}} = \int_{0}^{r(R)} \lambda \ud F_\lambda =: H \quad \text{in norm}.
\]
The non-zero eigenvalues of $H$ are precisely the moduli of the non-zero eigenvalues of $R$. Furthermore, the algebraic multiplicity of any non-zero eigenvalue $\mu > 0$ of $H$ is exactly the sum of the algebraic multiplicities of the eigenvalues of $R$ having modulus $\mu$.
}
\end{thm}

\begin{proof}
For the family $\{ F_\lambda \}_{\lambda \in \R}$, we first verify the limits at $\pm \infty$, monotonicity, and right-continuity.

\begin{itemize}
    \item[(i)] \textit{Limits at $-\infty$ and $\infty$}: By definition, $\bigwedge_{\lambda \in \R} F_\lambda = 0$. Since the Riesz idempotent $e_\lambda = I$ for all $\lambda \ge r(R)$, it follows that $\bigvee_{\lambda \in \R} F_\lambda = I$.
    \item[(ii)] \textit{Monotonicity}: Suppose $0 < \lambda \le \lambda'$. The containment of the spectral discs, $\D_\lambda \subseteq \D_{\lambda'}$, implies the containment of the ranges of the corresponding idempotents, $\ran(e_\lambda) \subseteq \ran(e_{\lambda'})$. This subset relation yields the operator inequality $F_\lambda \le F_{\lambda'}$.
    \item[(iii)] \textit{Right-continuity}: Since the non-zero spectrum $\sigma(R) \setminus \{0\}$ has no accumulation points in $\C \setminus \{0\}$, the set of its moduli, $|\sigma(R)| \setminus \{ 0 \}$, is a discrete subset of $(0, \infty)$. Then, for any given $\lambda > 0$, let $\delta = d(\lambda, |\sigma(R)| \setminus \{\lambda \})$ represent the distance to the nearest distinct spectral value. For any intermediate value $\lambda'$ strictly between $\lambda$ and $\lambda + \delta$, we have $\sigma(R) \cap \D_{\lambda'} = \sigma(R) \cap \D_\lambda$. Consequently, the Riesz idempotents remain identical ($e_{\lambda'} = e_{\lambda}$), which implies $F_{\lambda'} = F_{\lambda}$. Thus, this projection valued map $\lambda \mapsto F_\lambda$ is locally constant on the right, and we obtain $F_{\lambda} = \bigwedge_{\lambda' > \lambda} F_{\lambda'}$ for all $\lambda > 0$. For $\lambda \le 0$, right-continuity holds trivially by the explicit definition $F_\lambda = \bigwedge_{\lambda' > \lambda} F_{\lambda'}$. 
\end{itemize}
Hence, $\{ F_\lambda \}_{\lambda \in \R}$ is a bounded resolution of the identity, bounded in support by the spectral radius $r(R)$. 

\vskip0.1in

\noindent Let $0 < \ep < 1$, and define the integrals:
\[
    H_\ep := \int_{(\ep, r(R)]} \lambda \ud F_{\lambda} \quad \text{and} \quad H := \int_{[0, r(R)]} \lambda \ud F_\lambda.
\]
Let $q_\ep := I-e_\ep$. Because $R$ is a Riesz operator, $\sigma(R) \setminus \D_{\ep}$ is a finite set of eigenvalues with finite algebraic multiplicity. Therefore, the Riesz idempotent $I-e_{\ep}$ has a finite-dimensional range, making the cut-off operator $R q_\ep$ a finite-rank operator. Furthermore, $\sigma(R q_{\ep}) \subseteq \overline{(\D_{r(R)} \setminus \D_{\ep})} \cup \{0\}$.

Applying Proposition \ref{prop:finite-rank} to this finite-rank operator yields:
\begin{equation*}
    \lim_{n \to \infty} |(R q_{\ep})^n|^{\frac{1}{n}} = H_{\ep} = H(I-F_{\ep}) \quad \text{in norm}.
\end{equation*}

Thus there exists an $n_0(\ep) \in \N$ such that $\big\| |(R q_{\ep})^n|^{\frac{1}{n}} - H_{\ep} \big\| \le \ep$ for all $n \ge n_0(\ep)$. By Lemma \ref{lem:appx_ineq}, we can choose a sufficiently large integer $n(\ep) \ge n_0(\ep)$ such that:
\begin{equation}\label{eqn:appx_ineq}
    (1-\ep) |(R q_{\ep})^n|^{\frac{1}{n}} \le |R^n|^{\frac{1}{n}} \le (1+\ep) |(R q_{\ep})^n|^{\frac{1}{n}} + 4 \ep I \quad \text{for all } n \ge n(\ep).
\end{equation}

We now estimate the difference between $|R^n|^{\frac{1}{n}}$ and $H$. Clearly $0 \le H-H_{\ep} \le \ep I$. Utilizing the right-hand side of (\ref{eqn:appx_ineq}), we find that for all $n \ge n(\ep)$:
\begin{align*}
|R^n|^{\frac{1}{n}} - H &\le (1+\ep) \big(|(R q_{\ep})^n|^{\frac{1}{n}} - H_{\ep}\big) + (1+\ep)(H_{\ep}-H) + \ep H + 4 \ep I \\
&\le \ep(1+\ep)I + \ep \|H\| I + 4 \ep I \\
&\le (6+\|H\|) \ep I . 
\end{align*}

A similar argument using the left-hand side of (\ref{eqn:appx_ineq}) yields the corresponding lower bound, and for all $n \ge n(\ep)$ we have:
\begin{align*}
|R^n|^{\frac{1}{n}} - H &\ge (1-\ep) \big(|(R q_{\ep})^n|^{\frac{1}{n}} - H_{\ep}\big) + (1-\ep)(H_{\ep}-H) - \ep H \\
&\ge -\ep(1-\ep)I - \ep(1-\ep)I - \ep \|H\| I \\
&\ge -(2+\|H\|) \ep I.
\end{align*}

Combining these bounds, we conclude that $\big\| |R^n|^{\frac{1}{n}} - H \big\| \le (6 + \|H\|)\ep$ for all $n \ge n(\ep)$. This establishes the convergence:
\[
    \lim_{n \to \infty} |R^n|^{\frac{1}{n}} = H \quad \text{in norm}.
\]

For $\ep > 0$, $H_\ep$ is the spectral restriction of $H$ to the interval $(\ep, \infty)$ and is a finite-rank positive operator. Since $\ep > 0$ is arbitrary, the stated spectral and multiplicity properties of $H$ follow immediately from Proposition \ref{prop:finite-rank}.
\end{proof}

\begin{remark}
The existence of non-compact Riesz operators is well-established. For instance, the unilateral weighted shift $W$ on $\ell^2(\N)$ with orthonormal basis $\{e_n\}_{n \in \N}$, defined by $W e_n = w_n e_{n+1}$ with weights
\[
w_n = \begin{cases} 
1 & \text{if } n \text{ is even}, \\
\frac{1}{2^{(n+1)/2}} & \text{if } n \text{ is odd},
\end{cases}
\]
is quasinilpotent (hence a Riesz operator) but \textbf{not} compact. Moreover, the sum of this operator $W$ with the compact normal operator $N$ defined by $Ne_n = \frac{1}{n}e_n$ yields a Riesz operator $N+W$ that is neither compact nor quasinilpotent. Indeed, one can verify that $1 \in \sigma(N+W)$ by explicitly constructing an eigenvector $x = \sum_{n=1}^\infty x_n e_n \in \ell ^2(\N)$ corresponding to the eigenvalue $1$, where the coefficients are given by:
$$x_{2j} = \frac{2j}{2^{j(j+1)/2}} \quad \text{and} \quad x_{2j+1} = \frac{2j+1}{2^{j(j+1)/2}} \quad \text{for } j \ge 1.$$
Consequently, Theorem~\ref{thm:main} applies to a strictly broader class of operators and is stronger than \cite[Theorem~3.10]{bhat-bala}.
\end{remark}

\begin{cor}[see \cite{davis}]
\label{cor:davis}
\textsl{
Let $K \in \kB(\kH)$ be a compact operator, and let $\{\lambda_j\}_{j=1}^\infty$ be its eigenvalues ordered such that $|\lambda_1| \ge |\lambda_2| \ge \cdots$, where non-zero eigenvalues are repeated according to their algebraic multiplicities. If the total algebraic multiplicity of the non-zero eigenvalues is a finite number $\nu(K) \in \N$, let us set $\lambda_j = 0$ for all $j > \nu(K)$. Then for every $j \in \N$, we have
$$\lim_{n \to \infty} s_j(K^n)^{\frac{1}{n}} = |\lambda_j|.$$
}
\end{cor}

\begin{proof}
By Theorem \ref{thm:main}, the sequence $|K^n|^{\frac{1}{n}}$ converges in norm to the positive operator $H$. Note that the map $T \mapsto s_j(T)$ is continuous with respect to the operator norm (see \cite[Theorem 2]{Fan}, \cite[Theorem 1.7]{Simon}). Thus,
$$\lim_{n \to \infty} s_j(K^n)^{\frac{1}{n}} = \lim_{n \to \infty} s_j(|K^n|^{\frac{1}{n}}) = s_j(H).$$

Moreover, Theorem \ref{thm:main} explicitly establishes that the non-zero eigenvalues of $H$ are exactly the moduli of the non-zero eigenvalues of $K$, and the algebraic multiplicity of any eigenvalue $\mu > 0$ of $H$ is the sum of the algebraic multiplicities of the eigenvalues of $K$ having modulus $\mu$. 

Because the singular values of a positive operator are precisely its eigenvalues, when arranged in non-increasing order the $j^{\textrm{th}}$ singular value of $H$ matches the $j^{\textrm{th}}$ largest eigenvalue-modulus of $K$, when both sequences are enumerated according to their respective algebraic multiplicities.

If the total algebraic multiplicity of the non-zero eigenvalues of $K$ is a finite integer $\nu(K) \in \N$, then $H$ similarly has exactly $\nu(K)$ non-zero eigenvalues, meaning $s_j(H) = 0$ for all $j > \nu(K)$. Thus, in all cases, we obtain $s_j(H) = |\lambda_j|$ for all $j \in \N$, which completes the proof.
\end{proof}

\begin{example}
\label{ex:NPS_fail}
Since quasinilpotent operators are spectral, Riesz operators (being essentially quasinilpotent) form a natural subclass of essentially spectral operators. Given that the NPS converges in norm for Riesz operators as well as spectral operators, it is natural to ask whether this property extends to all essentially spectral operators. However, even for a rank-one perturbation of a spectral operator, norm convergence of its NPS is \textbf{not} guaranteed. 

Consider the Hilbert space $\kH = \ell^2(\Z)$ with its standard orthonormal basis $\{e_k\}_{k \in \Z}$. Let $T$ be the bilateral forward shift defined by $T e_k = e_{k+1}$, which is a unitary (and hence spectral) operator. Let $P_0$ be the orthogonal projection onto $\text{span}\{e_0\}$, and define the rank-one perturbation $K = -T P_0$. The perturbed operator $A = T + K$ is essentially spectral, acting as $A e_k = e_{k+1}$ for $k \neq 0$, and $A e_0 = 0$.

For any $n \in \N$, applying $A^n$ yields
\[
A^n e_k = 
\begin{cases} 
e_{k+n} & \text{if } k > 0 \text{ or } k \le -n, \\ 
0 & \text{if } -n < k \le 0. 
\end{cases}
\] Thus, $A^n$ is a partial isometry, and $P_n := |A^n|$ is the orthogonal projection onto the subspace $\text{span}\{e_k \mid k > 0 \text{ or } k \le -n\}$. Because $P_n$ is a projection, we have $|A^n|^{\frac{1}{n}} = P_n^{\frac{1}{n}} = P_n$.

For any $m > n$, observe that $P_m e_{-n} = 0$ while $P_n e_{-n} = e_{-n}$ . This implies that $\|(P_n - P_m)e_{-n}\| = 1$. Consequently, the sequence $\big\{|A^n|^{\frac{1}{n}}\big\}_{n \in \N}$ fails to converge in norm.
\end{example}

\section{Acknowledgements}
We would like to express our gratitude to Neeru Bala and B. V. Rajarama Bhat for helpful discussions. The second-named author gratefully acknowledges the financial support received through the ISI Institute fellowship during the initial phase of this research, and expresses sincere gratitude to Prof.\ Tirthankar Bhattacharya for his support through the J.C. Bose Fellowship JCB/2021/000041 of ANRF, at the Indian Institute of Science. 
\vskip 0.1in

\bibliographystyle{amsalpha}
\bibliography{references}

\end{document}